\documentclass[12pt,leqno,fleqn]{amsart}  
\usepackage{tikz}
\usepackage{amsmath,amstext,amsthm,amssymb,amsxtra}
\usepackage[top=1.5in, bottom=1.5in, left=1.25in, right=1.25in]	{geometry}
\usepackage{lmodern}

\usepackage{mathtools}
\mathtoolsset{showonlyrefs,showmanualtags}

\usepackage{hyperref} 
\hypersetup{
	colorlinks=true,       
	linkcolor=blue,          
	citecolor=magenta,        
	filecolor=magenta,      
	urlcolor=cyan           
}

\usepackage[msc-links]{amsrefs}

\newtheorem{theorem}{Theorem}[section]
\newtheorem{proposition}{Proposition}[section]
\newtheorem{lemma}{Lemma}[section]
\newtheorem{corollary}{Corollary}[section]
\newtheorem{OldTheorem}{Theorem}
\theoremstyle{definition}
\newtheorem{definition}{Definition}[section]
\theoremstyle{definition}

\theoremstyle{remark}

\numberwithin{equation}{section}

\title[] {Probability inequalities for multiplicative sequences of random variables}

\author{Grigori A. Karagulyan}
\address{Faculty of Mathematics and Mechanics, Yerevan State
	University, Alex Manoogian, 1, 0025, Yerevan, Armenia} 
\email{g.karagulyan@ysu.am}
\address{Institute of Mathematics of NAS of RA, Marshal Baghramian ave., 24/5, Yerevan, 0019, Armenia} 
\email{g.karagulyan@gmail.com}

\subjclass[2010]{42C05, 42C10, 42A55, 60G42, 60G48}
\keywords{multiplicative system, martingale difference, Khintchin inequality, Azuma-Hoeffding inequality, lacunary subsystem, Rademacher random variables}
	\begin{document}
\begin{abstract}
We extend some sharp inequalities for martingale-differences  to general multiplicative systems of random variables. The key ingredient in the proofs is a technique reducing the general case to the case of Rademacher random variables without change of the constants in inequalities. 
\end{abstract}

	\maketitle  
\section{Introduction}
	A sequence of bounded random variables  $\phi_n$, $n=1,2,\ldots$ (finite or infinite) is said to be multiplicative if the equality
	\begin{equation}\label{a13}
	{\textbf E}\left[\phi_{n_1}\phi_{n_2}\ldots \phi_{n_\nu}\right]=0
	\end{equation}
	holds for all possible choices of indexes $n_1< n_2< \ldots< n_\nu$.  Well-known examples of multiplicative sequences are mean zero independent random variables and more general, the martingale-differences, since the condition 
	\begin{equation*}
	{\textbf E}(\phi_n|\phi_1,\ldots,\phi_{n-1})=0
	\end{equation*}
	in the definition of the martingale-difference implies \eqref{a13}. The sequences $\{\sin(2^{k+1}\pi x)\}$ and $\{\sin(2n_k\pi x)\}$, where $n_k$ are integers satisfying $n_{k+1}\ge 3n_k$, are known to be non-martingale examples of a multiplicative systems on the unit interval $(0,1)$ (see \cite{Zyg}, chap. 5).  
	
	Note that multiplicative systems were introduced by Alexits in his famous monograph \cite {Alex}. It was proved by Alexits-Sharma \cite{AlexSh} that the uniformly bounded multiplicative systems are convergence systems. Recall that an infinite system of random variables $\{\phi_k\}$ is said to be a convergence system if the condition $\sum_k a_k^2<\infty$ implies almost sure convergence of series  $\sum_k a_k\phi_k$. Furthermore, this and other convergence properties of multiplicative type systems were generalized in the papers \cite{Gap1,Gap2,KoRe,Kom, Fuk1,Fuk2}. 
	
	Let ${\mathfrak M}$ be a family of nonempty subsets of ${\mathbb Z}_n=\{1,2,\ldots,n\}$, that is ${\mathfrak M}\subset 2^{{\mathbb Z}_n}\setminus \{\varnothing\}$. 
	A system of random variables $\phi=\{\phi_k:\, k=1,2,\ldots, n\}$ is said to be ${\mathfrak M}$-multiplicative if relation \eqref{a13} holds for all $\{n_1,n_2,\ldots,n_\nu\}\in {\mathfrak M}$. If ${\mathfrak M}=2^{{\mathbb Z}_n}\setminus \{\varnothing\}$, then $\phi$ turns to be a "full" multiplicative system. Likewise, $\phi$ is called ${\mathfrak M}$-independent if for any collection $\{n_1,n_2,\ldots,n_\nu\}\in {\mathfrak M}$ the members $\phi_{n_1}, \phi_{n_2},\ldots,\phi_{n_\nu}$ are scholastically independent. We will consider systems of bounded random variables $\phi=\{\phi_k:\, k=1,2,\ldots, n\}$ satisfying
	\begin{equation}\label{a28}
	A_k\le \phi_k\le B_k,  \text{ where }A_k<0<B_k.
	\end{equation}
	Setting $C_k=\min\{-A_k,B_k\}$, we define the multiplicative error of $\phi$ over a family of index subsets ${\mathfrak M}\subset 2^{{\mathbb Z}_n}\setminus \{\varnothing\}$ to be the quantity 
	\begin{equation}\label{a30}
	\mu=\mu(\phi,{\mathfrak M})=\sum_{\{n_1,n_2,\ldots,n_\nu\}\in {\mathfrak M}}\frac{1}{C_{n_1}C_{n_2}\ldots C_{n_\nu}}\left|{\textbf E}[\phi_{n_1} \phi_{n_2}\ldots \phi_{n_\nu}]\right|.
	\end{equation}
	For an integer $l\le n$ denote by ${\mathfrak M}_l$ the family of nonempty subsets of ${\mathbb Z}_n$ having cardinality $\le l$. If $l=n$, then we have ${\mathfrak M}_n=2^{{\mathbb Z}_n}\setminus \{\varnothing\}$.

	The results of the present paper provide a technique that may reduce the study of some properties of bounded multiplicative type systems to the case of Rademacher random variables. 
	\begin{theorem}\label{T1}
		Let $\Phi:{\mathbb R}\to {\mathbb R}^+$ be a convex function and $\phi=\{\phi_k:\, k=1,2,\ldots, n\}$ be a system  of random variables satisfying \eqref{a28}.  Then for any integer $l\le n$ and a choice of coefficients $a_1,\ldots, a_n$ it holds the inequality
		\begin{equation}\label{a12}
		\textbf{E}\left[\Phi\left(\sum_{k=1}^na_k\phi_k\right)\right]\le 	(1+\mu(\phi,{\mathfrak M}_l))\textbf{E}\left[\Phi\left(\sum_{k=1}^na_k\xi_k\right)\right],
		\end{equation}
		where $\xi_k$, $k=1,2,\ldots,n$ are $\{A_k,B_k\}$- valued mean zero ${\mathfrak M}_l$-independent random variables. 
	\end{theorem}
	Notice that if a system $\phi$ is ${\mathfrak M}_l$-multiplicative, then $\mu(\phi,{\mathfrak M}_l)=0$. So applying {Theorem \ref{T1} for ${\mathfrak M}_l$-multiplicative systems with the parameters $A_k=-1$ and $B_k=1$, we immediately obtain the following.
	\begin{corollary}\label{C1}
		Let $\Phi:{\mathbb R}\to {\mathbb R}^+$ be a convex function. If $\phi=\{\phi_k:\, k=1,2,\ldots, n\}$, is a system of ${\mathfrak M}_l$-multiplicative ($l\le n$) random variables satisfying $\|\phi_k\|_\infty\le 1$, then for any choice of coefficients $a_1,\ldots, a_n$ we have
		\begin{equation}\label{a33}
		\textbf{E}\left[\Phi\left(\sum_{k=1}^na_k\phi_k\right)\right]\le {\textbf E}\left[\Phi\left(\sum_{k=1}^na_kr_k\right)\right],
		\end{equation}
		where $r_k$, $k=1,2,\ldots,n$, are Rademacher ${\mathfrak M}_l$-independent random variables.
	\end{corollary}
	Koml\'{o}s \cite{Kom} and Gaposhkin \cite{Gap2} independently proved that
	\begin{OldTheorem}[Koml\'{o}s-Gaposhkin]\label{KG}
		If an infinite sequence of random variables $\phi=\{\phi_n\}$ satisfies condition \eqref{a13} for a fixed even integer $\nu >2$ and the norms $\|\phi_n\|_\nu$ are uniformly bounded, then $\{\phi_k\}$ is a convergence system.
	\end{OldTheorem}
	Moreover, the papers \cite{Kom, Gap2} in fact prove a Khintchin type inequality 
	\begin{equation}\label{a14}
	\left\|\sum_{k=1}^n a_k\phi_k\right\|_\nu \le 	K(\nu) \left(\sum_{k=1}^na_k^2\right)^{1/2},
	\end{equation}
	that implies Theorem \ref{KG} according to a well-known result due to Stechkin (see \cite{KaSa}, chap 9.4). On the other hand none of those papers provide an  estimation for the Khintchin constant $K(\nu)$. A careful examination of paper \cite{Kom} may provide only $K(\nu)\lesssim \nu$ even if the norms $\|\phi_k\|_\infty$ are uniformly bounded. While for many classical examples of multiplicative systems it holds the bound $K(\nu)\lesssim \sqrt \nu$. For lacunary trigonometric systems $\sin(2\pi n_kx)$, $n_{k+1}>\lambda n_k$, $\lambda>1$, such a bound is due to Zygmund (see \cite{Zyg}, chap. 5), for the uniformly bounded martingale-differences it follows from the Azuma-Hoeffding inequality \cite{Azu, Hoe}. In the case of Rademacher independent random variables the Khintchin inequality holds with the constant 
	\begin{equation}\label{a17}
	K(p)=2^{1/2}\left(\Gamma((p+1)/2)/\pi)\right)^{1/p}, \quad p>2,
	\end{equation}
	which is known to be optimal (see \cite{Haa}, \cite{Ste}, \cite{Young}). Using Corollary \ref{C1}, the Khintchin sharp inequality for Rademacher independent random variables can be extended to general uniformly bounded multiplicative systems.
	
\begin{corollary}\label{T3}
	If a system of random variables $\phi=\{\phi_k:\, k=1,2,\ldots, n\}$ is multiplicative and $\|\phi_k\|_\infty\le 1$, then for any choice of coefficients $a_1,\ldots ,a_n$ we have
	\begin{equation}\label{a16}
	\left\|\sum_{k=1}^na_k\phi_k\right\|_p\le K(p)\left(\sum_{k=1}^na_k^2\right)^{1/2},\quad p>2,
	\end{equation}
	where $K(p)$ is the optimal constant from \eqref{a17}.
\end{corollary}
It is well known that the classical proof of the Khintchin inequality for even integers $p$ only ${\mathfrak M}_p$-independence of Rademacher functions is used. It is also known that in this case the Khintchin optimal constant is $((p-1)!!)^{1/p}$ (see \cite{Ste} or \cite{Shi} chap 2). So once again applying Corollary \ref{C1}, we can extend this result to the following inequality.
\begin{corollary}
	Let $\phi=\{\phi_k:\, k=1,2,\ldots, n\}$ be a ${\mathfrak M}_p$-multiplicative system such that $\|\phi_k\|_\infty\le 1$ and $p$ is an even integer with $2\le p\le n$. Then we have
	\begin{equation}\label{a20}
	\left\|\sum_{k=1}^na_k\phi_k\right\|_p\le K(p)\cdot \left(\sum_{k=1}^na_k^2\right)^{1/2},
	\end{equation}
	with the optimal constant $K(p)=((p-1)!!)^{1/p}$.
\end{corollary}
Applying Theorem \ref{T1}, we also prove the following generalization of a well-known martingale inequality due to Azuma-Hoeffding \cite{Azu, Hoe}. 
\begin{theorem}\label{T2}
	If a system  of random variables $\phi=\{\phi_k:\, k=1,2,\ldots, n\}$ satisfies \eqref{a28}, then it holds the inequality
	\begin{equation}\label{a21}
	\left|\left\{\sum_{k=1}^n\phi_k>\lambda\right\}\right|\le (1+\mu(\phi,{\mathfrak M}_n))\exp\left(-\frac{2\lambda^2}{\sum_{k=1}^n(B_k-A_k)^2}\right),\quad \lambda>0.
	\end{equation}
\end{theorem}
\begin{corollary}\label{C6}
	If $\phi=\{\phi_k:\, k=1,2,\ldots, n\}$ is multiplicative and satisfies \eqref{a28}, then
	\begin{equation}\label{a34}
	\left|\left\{\sum_{k=1}^n\phi_k>\lambda\right\}\right|\le \exp\left(-\frac{2\lambda^2}{\sum_{k=1}^n(B_k-A_k)^2}\right),\quad \lambda>0.
	\end{equation}
\end{corollary}

	In the proof of the main result we use some arguments used  in the papers \cite{Kar1, Kar2} that is a transformation of a bounded orthogonal system into a system of two-valued functions in the study of certain problems. 1) First, we extend functions $\phi_k$ up to the interval $[0,1+\mu)$, where the new system becomes ${\mathfrak M}_l$-multiplicative (Lemma \ref{L0}), 2)  an approximation procedure applied by Lemma \ref{L1} reduces our problem to the case of step functions, 3) and those after a certain transformation procedure, giving by Lemma \ref{L2}, produce ${\mathfrak M}_l$-multiplicative system of $\{A_k,B_k\}$-valued functions,
4) a dilation of the interval $[0,1+\mu)$ back to $[0,1)$ finalizes the proof, generating the constant $1+\mu$ in \eqref{a12}, since any ${\mathfrak M}_l$-multiplicative system of two-valued functions is ${\mathfrak M}_l$-independent (Lemma \ref{L3}).

The paper is organized as follows. In Section \ref{S2} we have collected preliminary lemmas. Section \ref{S3} provides the proofs of the main results. In Section \ref{S4} we give applications concerning sub-Gausian estimations for lacunary subsequences of trigonometric systems.

\section{Preliminary lemmas}\label{S2}
A real-valued  function $f$ defined on $[a,b)$ is said to be a step function if it can be written as a finite linear combination of indicator functions of intervals $[\alpha,\beta)\subset [a,b)$.
\begin{lemma}\label{00}
	For any integer $n\ge 2$ and interval $I=[a,b)$ there exists a system of unimodular step functions $f_k(x)$, $k=1,2,\ldots,n$, on $I$ such that 
	\begin{align}
		&\prod_{k=1}^nf_k(x)=1,\quad x\in I,\label{x2}\\
		&\int_I\left(\prod_{k\in U}f_k\right)=0.\label{x1}
	\end{align}
for any nonempty $U\subsetneq \{1,2,\ldots, n\}$.
\end{lemma}
\begin{proof}
	Suppose unimodular step functions $f_k$, $k=1,2,\ldots,n$, satisfy \eqref{x2} and let $V\subsetneq \{1,2,\ldots, n\}$ be nonempty. Chose $m\in \{1,2,\ldots,n\}\setminus V$ and $l\in V$ arbitrarily. Let $J\subset I$ be a maximal constancy interval of functions $f_k$, $J^-$ and $J^+$ be the left and right halves of $J$. We redefine $f_l$ and $f_m$, changing their signs on $I^+$. We do so with each interval of constancy. Clearly, at the end of this procedure we will have 
	\begin{equation*}
		\int_I\left(\prod_{k\in V}f_k\right)=0.
	\end{equation*}
Moreover, one can check if \eqref{x1} is satisfied for some $U$ before the reconstruction of the functions, then so we will also have after the reconstruction. Starting with functions $f_k(x)=1$, $k=1,2,\ldots$, we apply this procedure for every $U\subsetneq \{1,2,\ldots, n\}$. In this way we get functions satisfying the conditions of lemma.
\end{proof}
\begin{lemma}\label{L0}
	Let $\phi=\{\phi_k:\, k=1,2,\ldots,n\}$ be a system of measurable functions on $[0,1)$ satisfying \eqref{a28} and ${\mathfrak M}$ be an arbitrary family of subsets of ${\mathbb Z}_n$. Then the functions $\phi_k$ can be extended up to the interval $[0,1+\mu)$ such that
	\begin{align}
	&A_k\le \phi_k(x)\le B_k\text { if } x\in [0,1+\mu),\label{a26}\\
	&\int_0^{1+\mu}\phi_{n_1}\phi_{n_2}\ldots \phi_{n_\nu}=0\text { for all } \{n_1,\ldots,n_\nu\}\in {\mathfrak M},\label{a27}
	\end{align}
	where $\mu=\mu(\phi,{\mathfrak M})$ is the multiplicative error \eqref{a30}.  Moreover, each $\phi_k$ is a step function on $[1,1+\mu)$.
	\end{lemma}
\begin{proof}
Set
	\begin{equation}\label{a51}
	\delta_{n_1,n_2,\ldots,n_\nu}=\frac{1}{C_{n_1}C_{n_2}\ldots C_{n_\nu}}\left|\int_0^1\phi_{n_1} \phi_{n_2}\ldots \phi_{n_\nu}\right|.
	\end{equation}
Divide $[1,1+\mu)$ into intervals $I_{n_1,n_2,\ldots,n_\nu}$ of lengths $\delta_{n_1,n_2,\ldots,n_\nu}$ considering only the collections $\{n_1, n_2, \ldots, n_\nu\}\in {\mathfrak M}$. We define functions $\phi_m$ on such an interval $I=I_{n_1,n_2,\ldots,n_\nu}$ as follows. If $m\notin \{n_1,n_2,\ldots ,n_\nu\}$, then we let $\phi_n= 0$ on $I$. Applying Lemma \ref{00}, one can define the functions $\phi_{n_1}, \phi_{n_2},\ldots,\phi_{n_\nu}$ such that $|\phi_{n_j}(x)|=	C_{n_j},\, x\in I,\quad j=1,2,\ldots,\nu,$ and 
	\begin{align}
 &\prod_{j=1}^\nu\phi_{n_j}(x)=-C_{n_1}\ldots C_{n_\nu}\cdot {\rm sign\,}\left(\int_0^1\phi_{n_1} \phi_{n_2}\ldots \phi_{n_\nu}\right),\quad x\in I,\label{x3}\\
 &\int_I\left(\prod_{j\in U}\phi_{j}\right)=0,\quad U\subsetneq\{n_1,n_2,\ldots,n_\nu\}.\label{x4}
	\end{align}
Obviously, this correctly determines the functions $\phi_k$  on $ [1,1+\mu)$ and those satisfy $A_k\le \phi_k(t)\le B_k$ for all $t\in [0,1+\mu)$. From  \eqref{a51} and \eqref{x3} it follows that
	\begin{equation}\label{a23}
	\int_{I}\phi_{n_1}\phi_{n_2} \ldots \phi_{n_\nu}=-\int_0^1\phi_{n_1} \phi_{n_2}\ldots \phi_{n_\nu}.
	\end{equation}
	One can also check that
	\begin{equation}\label{x10}
			\int_{I_{m_1,\ldots,m_l}}\phi_{n_1}\phi_{n_2}\ldots \phi_{n_\nu}=0,\text { if }\{m_1,\ldots,m_l\}\neq \{n_1,n_2,\ldots ,n_\nu\}.
	\end{equation}
Indeed, if there is a $n_k\notin \{m_1,\ldots,m_l\}$, then by definition $\phi_{n_k}=0$ on $I_{m_1,\ldots,m_l}$ and \eqref{x10} follows. Otherwise we will have $\{n_1,n_2,\ldots ,n_\nu\}\subsetneq \{m_1,\ldots,m_l\}$ and \eqref{x10} follows from \eqref{x4}. By \eqref{x10} we get
	\begin{equation}\label{x5}
		\int_{[1,1+\mu)\setminus I}\phi_{n_1}\phi_{n_2}\ldots \phi_{n_\nu}=0.
	\end{equation}
	From \eqref{a23} and \eqref{x5} we obtain
	\begin{equation}
	\int_0^{1+\mu }\phi_{n_1}\phi_{n_2}\ldots  \phi_{n_\nu}=0, \quad \{n_1, n_2, \ldots, n_\nu\}\in {\mathfrak M},
	\end{equation}
	which completes the proof of lemma.	
\end{proof}

\begin{lemma}\label{L1}
	Let measurable functions $\phi_k$, $k=1,2,\ldots,n$, defined on $[a,b)$, satisfy \eqref{a28}. Then for any $\delta>0$ one can find step functions $f_k$, $k=1,2,\ldots, n$, on  $[a,b)$ with $A_k\le f_k\le B_k$ such that
	\begin{equation}\label{a7}
	|\{|\phi_k-f_k|>\delta\}|<\delta,\quad k=1,2,\ldots,n,
	\end{equation}
	and 
	\begin{equation}\label{a31}
	\int_a^bf_{n_1}f_{n_2}\ldots f_{n_\nu}=0\text { as } \{n_1,n_2,\ldots,n_\nu\}\in {\mathfrak M},
	\end{equation}
where ${\mathfrak M}$ is the family of collections $\{n_1,n_2,\ldots,n_\nu\}$ satisfying \eqref{a13}.
\end{lemma}
\begin{proof}
Without loss of generality we can suppose that $[a,b)=[0,1)$. For any $\varepsilon >0$ one can find step functions $u=\{u_k\}$, $A_k\le u_k\le B_k$, such that 
	\begin{align}
	&\int_0^1|\phi_k-u_k|<\varepsilon,\label{x6}\\
	&\left|\int_0^1u_{n_1} u_{n_2}\ldots u_{n_\nu}\right|<\varepsilon\text { as } \{n_1,n_2,\ldots,n_\nu\}\in {\mathfrak M}.\label{a6}
	\end{align}
	Let $\mu=\mu(u,{\mathfrak M})$ be the multiplicative error of the system $u$. Applying Lemma \ref{L0}, the functions $u_k$ can be extend to the step functions on $[0,1+\mu)$ such that  $A_k\le u_k(t)\le B_k$, $t\in [0,1+\mu)$, and
	\begin{equation*}
	\int_0^{1+\mu}u_{n_1}u_{n_2}\ldots u_{n_\nu}=0\text { for all } \{n_1,\ldots,n_\nu\}\in {\mathfrak M}.
	\end{equation*}	
	Set
	\begin{equation}
	f_k(x)=u_k\left((1+\mu)x\right),\quad x\in [0,1).
	\end{equation}
	Obviously, \eqref{a31} will be satisfied. Then, using \eqref{a6} one can get a small $\mu$ by choosing a small enough $\varepsilon$. So by \eqref{x6} we can write
	\begin{align}
	\int_0^1|f_k(x)-\phi_k(x)|dx&\le \int_0^1|f_k(x)-u_k(x)|dx+\int_0^1|\phi_k(x)-u_k(x)|dx\\
	&<\int_0^1|u_k\left((1+\mu)x\right)-u_k(x)|dx+\varepsilon <\delta^2.
	\end{align}
	Applying Chebyshev's inequality, we get \eqref{a7}. Lemma is proved. 
\end{proof}

\begin{lemma}\label{L2}
	Let $f_1,f_2,\ldots,f_n$ be real-valued step functions on $[a,b)$ satisfying $A_k\le f_k(x)\le B_k$. Then there are $\{A_k,B_k\}$-valued step functions $g_1,g_2,\ldots,g_n$ on $[a,b)$ such that 
\begin{equation}\label{a8}
	\int_a^bg_{n_1} g_{n_2} \ldots g_{n_\nu}=\int_a^bf_{n_1}f_{n_2}\ldots  f_{n_\nu},
\end{equation}
for any choice of $\{n_1,n_2,\ldots,n_\nu\}\in {\mathbb Z}_n$, and for any convex function $\Phi:{\mathbb R}\to {\mathbb R}^+$ it holds the inequality
\begin{equation}\label{a24}
\int_a^b\Phi\left(\sum_{k=1}^na_kf_k\right)\le \int_a^b\Phi\left(\sum_{k=1}^na_kg_k\right)
\end{equation}
for any coefficients $a_k$.
\end{lemma}
\begin{proof}
	Let $\Delta_j$, $j=1,2,\ldots,m$ be the intervals of constancy of functions $f_k$. Let $\Delta=[\alpha,\beta)$ be one of those intervals. Observe that the point
	\begin{equation}
	c=\frac{B_1\alpha-A_1\beta}{B_1-A_1}+\frac{1}{B_1-A_1}\cdot \int_\alpha^\beta f_1(t)dt
	\end{equation}
is in the closure of the interval $\Delta$. Then we define $g_1$ on $\Delta$ as
\begin{equation}
g_1(x)=B_1\cdot {\textbf 1}_{ [\alpha,c)}(x)+A_1\cdot{\textbf 1}_{[c,\beta)}(x),\quad x\in \Delta.
\end{equation}
Applying this to each $\Delta_j$, we will have $g_1$ defined on entire $[a,b)$ and one can check 
\begin{equation}\label{a3}
\int_{\Delta_j}g_1(t)dt=\int_{\Delta_j}f_1(t)dt,\quad j=1,2,\ldots,m.
\end{equation}
Since each $f_k$ is constant on the intervals $\Delta_j$, $j=1,2,\ldots,m$, from \eqref{a3} we conclude that
	\begin{equation}\label{a4}
\int_0^1g_{1} f_{n_2} \ldots f_{n_\nu}=\int_0^1f_{1}f_{n_2}\ldots f_{n_\nu},
\end{equation}
for any collection $1<n_2<\ldots<n_\nu$. We also claim that
\begin{equation}\label{a5}
\int_a^b\Phi\left(\sum_{k=1}^na_kf_k\right)\le \int_a^b\Phi\left(a_1g_1+\sum_{k=2}^na_kf_k\right).
\end{equation}
Fix an interval $\Delta_j$ and suppose that $f_k(t)=c_k$ on $\Delta_j$. Applying \eqref{a3} and the Jessen inequality, we get
\begin{align}
\int_{{\Delta_j}}\Phi\left(\sum_{k=1}^na_kf_k(t)\right)dt&=\Phi\left(\sum_{k=1}^na_kc_k\right)|{\Delta_j}|\\
&=\Phi\left(\frac{1}{|{\Delta_j}|}\int_{{\Delta_j}}a_1g_1(t)dt+\sum_{k=2}^na_kc_k\right) |{\Delta_j}|\\
&=\Phi\left(\frac{1}{|{\Delta_j}|}\int_{{\Delta_j}}\left(a_1g_1(t)+\sum_{k=2}^na_kc_k\right)dt\right)|{\Delta_j}|\\
&\le \int_{{\Delta_j}}\Phi\left(a_1g_1+\sum_{k=2}^na_kf_k\right)dt,
\end{align}
then the summation over $j$ implies \eqref{a5}. Applying the same procedure to the new system 
$g_1,f_2,\ldots, f_n$ we can similarly replace $f_2$ by $g_2$. Continuing this procedure we will replace all functions $f_k$ to $g_k$ ensuring the conditions of lemma.
\end{proof}
\begin{lemma}\label{L3}
If $g_k$, $k=1,2,\ldots,n$, is a ${\mathfrak M}_l$-multiplicative system of nonzero random variables such that each $g_k$ takes two values, then $g_k$ are ${\mathfrak M}_l$-independent.
\end{lemma}
\begin{proof}
	Suppose that  $g_k$ takes values $A_k$ and $B_k$. Since ${\textbf E}(g_k)=0$, we can say $A_k<0<B_k$ and 
	\begin{equation}\label{x7}
	|\{g_k(x)=A_k\}|=\frac{B_k}{B_k-A_k},\quad |\{g_k(x)=B_k\}|=\frac{A_k}{A_k-B_k}.
	\end{equation}
	Let $C_k$ be a sequence that randomly equal either $A_k$ or $B_k$. We need to prove
	\begin{equation*}
		|\{g_j=C_j:\, j\in M\}|=\prod_{j\in M}	|\{g_j=C_j\}|
	\end{equation*}
for any $M\in {\mathfrak M}_l$. Without loss of generality we can suppose that $C_j=A_j$ for all $j\in M$. Then, using the multiplicative condition and \eqref{x7}, we obtain	
	\begin{align*}
	|\{g_j(x)&=A_k:\, j\in M\}|={\textbf E}\left[\prod_{j\in M}\left(\frac{B_j}{B_j-A_j}-\frac{g_j}{B_j-A_j}\right)\right]\\
	&=\prod_{j\in M}\frac{B_j}{B_j-A_j}\\
	&\qquad +\sum_{D\subsetneq M}\quad (-1)^{{\rm card}(D)}\bigg[\prod_{j\in M\setminus D}\frac{B_j}{B_j-A_j}\\
	&\qquad\times \prod_{j\in D}\frac{1}{B_{j}-A_{j}}\cdot {\textbf E}\left(\prod_{j\in D}g_j\right)\bigg]\\
	&=\prod_{j\in M}|\{g_{j}(x)=A_{j}\}|,
	\end{align*}
completing the proof of lemma.
\end{proof}

\section{Proof of Theorems}\label{S3}
\begin{proof}[Proof of Theorem \ref{T1}]
	Without loss of generality we can suppose that $\phi_k$ are defined on $[0,1)$. Applying Lemma \ref{L0} for ${\mathfrak M}={\mathfrak M}_l$, we can find extensions of $\phi_k$ up to $[0,1+\mu)$, satisfying  \eqref{a26} and \eqref{a27} (with ${\mathfrak M}={\mathfrak M}_l$). By \eqref{a27} $\phi_k$ turns to be ${\mathfrak M}_l$-multiplicative on $[0,1+\mu)$. Applying Lemma \ref{L1}, we find an ${\mathfrak M}_l$-multiplicative system of step functions $f_k$ on $[0,1+\mu)$ satisfying 
	\begin{equation}\label{a32}
	|\{x\in [0,1+\mu):\,|f_k(x)-\phi_k(x)|>\delta\}|<\delta.
	\end{equation}
	Finally, we apply Lemma \ref{L2} and get $A_k,B_k$-valued step functions $g_k$ defined on $[0,1+\mu)$ and satisfying \eqref{a8} and \eqref{a24} ($a=0,b=1+\mu$). Since $\{f_k\}$ is ${\mathfrak M}_l$-multiplicative, in view of \eqref{a8} so we will have for $\{g_k\}$. By Lemma \ref{L3} functions $\xi_k(x)=g_k((1+\mu)x)$ turn to be an $A_k,B_k$-valued ${\mathfrak M}_l$-independent random variables. Observe that
	\begin{equation}
	\int_0^1\Phi\left(\sum_{k=1}^na_k\phi_k\right)\le \varepsilon(\delta, \{a_k\}) +\int_0^1\Phi\left(\sum_{k=1}^na_kf_k\right),
	\end{equation}
where by \eqref{a32} we have
\begin{equation}\label{x9}
	\varepsilon =\varepsilon(\delta, \{a_k\})\to 0\text{ as } \delta\to 0, 
\end{equation}
and it holds for any choice of coefficients $\{a_k:\, k=1,2,\ldots\}$. Thus, from \eqref{a24} we obtain
	\begin{align}
	\int_0^1\Phi\left(\sum_{k=1}^na_k\phi_k\right)&\le \varepsilon +\int_0^1\Phi\left(\sum_{k=1}^na_kf_k\right)\le \varepsilon +\int_0^{1+\mu}\Phi\left(\sum_{k=1}^na_kf_k\right)\\
	&\le \varepsilon +\int_0^{1+\mu}\Phi\left(\sum_{k=1}^na_kg_k\right)\\
	&=\varepsilon +(1+\mu)\int_0^{1}\Phi\left(\sum_{k=1}^na_kg_k((1+\mu)x)\right)dx\\
	&= \varepsilon +	(1+\mu)\textbf{E}\left[\Phi\left(\sum_{k=1}^na_k\xi_k\right)\right].
	\end{align}
	Therefore, taking into account \eqref{x9}, one can easily get \eqref{a12}. Indeed, the system $\{\xi_k\}$ depends on $\delta$ and is independent of $\{a_k\}$. So let $\{\xi_k^{(m)}\}$ be the system corresponding to $\delta =1/m$. We can suppose that there is a partition $0=x_0^{(m)}\le x_1^{(m)}\le, \ldots, x_{2^n}^{(m)}=1$ such that each function $\xi_k^{(m)}$ is equal to a constant (either $A_k$ or $B_k$) on every interval $[x_{j-1}^{(m)},x_j^{(m)})$, $1\le j\le 2^n$ and this constant is independent of $m$. Then we find a sequence $m_k$ such that each sequence $x_j^{(m_k)}$, $k=1,2,\ldots $ is convergence. One can check that this  generates a limit  sequence of ${\mathfrak M}_l$-multiplicative $\{A_k,B_k\}$-valued random variables $\xi_k$, $k=1,2,\ldots,n$, satisfying \eqref{a12}.
\end{proof}
\begin{proof}[Proof of Theorem \ref{T2}]
Applying Theorem \ref{T1} for $\Phi(t)=\exp(\gamma t)$, $\gamma>0$, and $l=n$, we find $\{A_k,B_k\}$-valued independent system $\{\xi_k\}$ satisfying \eqref{a12}. Thus, we get
\begin{align*}
\left|\left\{\sum_{k=1}^n\phi_k>\lambda\right\}\right|&\le e^{-\gamma \lambda}	{\textbf E}\left[\exp\left(\gamma\sum_{k=1}^n\phi_k\right)\right]\\
&\le (1+\mu)e^{-\gamma \lambda} {\textbf E}\left[\exp\left(\gamma\sum_{k=1}^n\xi_k\right)\right]\\
&=(1+\mu)e^{-\gamma \lambda}\prod_{k=1}^n	{\textbf E}\left[\exp\left(\gamma \xi_k\right)\right].
\end{align*}
Then applying Hoeffding's \cite{Hoe} inequality we get
\begin{equation*}
	{\textbf E}\left[\exp\left(\gamma  \xi_k\right)\right]\le \exp\left(\frac{\gamma^2(B_k-A_k)^2}{8}\right)
\end{equation*}
and finally,
\begin{equation*}
\left|\left\{\sum_{k=1}^n\phi_k>\lambda\right\}\right|\le (1+\mu) e^{-\gamma \lambda}\exp\left(\frac{\gamma^2\sum_{k=1}^n(B_k-A_k)^2}{8}\right).
\end{equation*}
Choosing $\gamma=\frac{4\lambda}{\sum_{k=1}^n(B_k-A_k)^2}$ we get the bound
\begin{equation*}
\left|\left\{\sum_{k=1}^n\phi_k>\lambda\right\}\right|\le (1+\mu)\exp\left(-\frac{2\lambda^2}{\sum_{k=1}^n(B_k-A_k)^2}\right),
\end{equation*}
which completes the proof of theorem.
\end{proof}
\section{Lacunary subsystems}\label{S4}
In this section we provide applications of the main results in lacunary systems.
\begin{definition}
	An infinite sequence of random variables $\phi=\{\phi_k:\, k=1,2,\ldots \}$ satisfying \eqref{a28} with $A_k=-1,B_k=1$ is said to be quasi-multiplicative if the multiplicative error \eqref{a30} is finite, that is $\mu=\mu(\phi,{\mathfrak M}_\infty)<\infty$, where ${\mathfrak M}_\infty$ denotes the family of all finite nonempty subsets of positive integers ${\mathbb Z}_+=\{1,2,\ldots\}$.
\end{definition}

\begin{definition}
	An infinite sequence of random variables $\phi=\{\phi_k:\, k=1,2,\ldots \}$ is said to be sub-Gaussian if there are constants $c_1,c_2>0$ such that the inequality 
	\begin{equation}
	{\textbf E}\left[\exp \left(c_1\frac{|\sum_{k=1}^na_k\phi_k|^2}{\sum_{k=1}^na_k^2}\right)\right]\le c_2
	\end{equation}
	holds for any coefficients $a_k$, $k=1,2,\ldots, n$.
\end{definition}
It is well-known that an equivalent condition for a system $\phi=\{\phi_k:\, k=1,2,\ldots \}$ to be sub-Gaussian is the bound $K(\nu)\lesssim \sqrt \nu$ in the Khintchin inequality \eqref{a14}. The following corollary, in addition to classical results stated in the introduction, provides new examples of sub-Gaussian sequences of random variables. It immediately follows from \eqref{a12} and the Khintchin inequality for the Rademacher independent random variables. Namely,
\begin{corollary}\label{C4}
Any quasi-multiplicative sequence of random variables is sub-Gaussian. 
\end{corollary}
We consider lacunary trigonometric system 
\begin{equation}\label{b10}
t_k(x)=\sin(2\pi \tau(k)x),\quad x\in [0,1),\quad k=1,2,\ldots,
\end{equation}
where
\begin{equation}\label{b9}
\tau(1)\ge 1,\quad \tau(k+1)>\lambda \tau(k),\quad k=1,2,\ldots,
\end{equation}
for some constant $\lambda>1$. In the case of integer $\tau(k)$ this sequence is known to be either multiplicative (if $\lambda\ge 3$) or finite union of multiplicative systems. It was proved by Zygmund (\cite{Zyg}, chap. 5) that in that case the system $t_k(x)$ is sub-Gaussian. Using Corollary \ref{C4} and the following result we prove that $t_k(x)$ is sub-Gaussian in the general case.
\begin{proposition}\label{P4}
If $\lambda>2$, then the sequence \eqref{b10} is quasi-multiplicative (with parameters $A_k=-1,B_k=1$ in \eqref{a30}). Moreover, the multiplicative error satisfies
\begin{equation*}
\mu=\mu(\phi,{\mathfrak M}_\infty)\le \frac{\lambda(\lambda -1)}{\pi (\lambda-2)^2}.
\end{equation*}
\end{proposition}
\begin{proof}
Let 
\begin{equation}\label{a35}
\{n_1,n_2,\ldots,n_\nu\}\in {\mathfrak M}_\infty
\end{equation}
be arbitrary collection of indexes with the head $n_\nu$, that is $n_1<n_2<\ldots<n_\nu$. 
Using the product to sum formulas for trigonometric functions, we can write the integral
	\begin{align}
	\int_0^1\prod_{j=1}^\nu \sin(2\pi \tau(n_j)x)dx,\label{b2}
		\end{align}
as an arithmetic mean of $2^{\nu-1}$ integrals of the forms
\begin{align}
\int_0^1\sin \left[(2\pi (\tau(n_\nu)\pm \tau(n_{\nu-1})\pm\ldots\pm \tau(n_{1}))x\right]dx,\label{b11}\\
\int_0^1\cos \left[(2\pi (\tau(n_\nu)\pm \tau(n_{\nu-1})\pm\ldots\pm \tau(n_{1}))x\right]dx.\label{b12}
\end{align}
A simple calculation shows that 
\begin{equation}\label{b1}
\tau(n_\nu)\pm \tau(n_{\nu-1})\pm\ldots\pm \tau(n_{1})\in \left(\frac{(\lambda-2)\tau(n_\nu)}{\lambda-1},\frac{\lambda\tau(n_\nu)}{\lambda-1}\right)
\end{equation}	
for all choices of $\pm$, so the absolute value of each integral in \eqref{b11} and \eqref{b12} can be estimated by $\frac{\lambda -1}{\pi (\lambda-2)\tau(n_\nu)}$. Thus the same bound we will have for the integral \eqref{b2}. Namely,
	\begin{equation*}
	\left|\int_0^1\prod_{j=1}^\nu \sin(2\pi \tau(n_j)x)\right|\le \frac{\lambda -1}{\pi (\lambda-2)\tau(n_\nu)}.
	\end{equation*}
On the other hand the number of collections \eqref{a35} with a fixed head  $n=n_\nu$ is equal to $2^{n-1}$ and we have $\tau(n)\ge \lambda^{n-1}$. Thus for the multiplicative error we obtain
	\begin{equation*}
	\mu\le  \frac{\lambda -1}{\pi (\lambda-2)}\sum_{n=1}^\infty\frac{2^{n-1}}{\tau(n)}\le \frac{\lambda(\lambda -1)}{\pi (\lambda-2)^2}.
	\end{equation*}

\end{proof}
The following corollary is an extension of the above mentioned theorem of Zygmund, where an integer sequence $\tau(k)$ is considered (see \cite{Zyg}, chap. 5, Theorem 8.20). Note that if $\tau(k)$ are integers, then the system \eqref{b10} becomes orthogonal and that is essential in the proof of the Zygmund theorem. Our proof of Corollary \ref{C5} essentially uses the main result via Corollary \ref{T3} and Corollary \ref{C4}. Hence,
	\begin{corollary}\label{C5}
	If real numbers  $\tau(k)$ satisfy \eqref{b9} with a constant $\lambda>1$, then lacunary system \eqref{b10} is sub-Gaussian.
\end{corollary}
\begin{proof}
	If the lacunarity order $\lambda\ge 3$, then by Proposition \ref{P4} our system $\{t_k(x)\}$ is a quasi-multiplicative and so sub-Gaussian by Corollary \ref{C4}. If $1<\lambda<3$, then \eqref{b10} can be split into $\lceil\log_\lambda 3\rceil$ number of systems of lacunarity order greater that $3$. It remains just notice that a finite union of sub-Gaussian sequences is sub-Gaussian. 
\end{proof}
\begin{definition}
	An infinite  sequence of random variables $\phi=\{\phi_k:\, k=1,2,\ldots \}$  is said to be unconditional convergence system if under the condition 
	$\sum_{n=1}^\infty a_n^2<\infty$ the series 
	\begin{equation}\label{b6}
	\sum_{k=1}^\infty a_k\phi_k(x)
	\end{equation}
	converges a.e. after any rearrangements of the terms.
\end{definition}
\begin{corollary}
	If an infinite sequence of real numbers  $\tau(k)$ satisfies \eqref{b9} with $\lambda>1$, then the lacunary system \eqref{b10} is an unconditional convergence system.
\end{corollary}
\begin{proof}
	According to  Corollary \ref{C5}, $\{t_k(x)\}$ is sub-Gaussian so satisfies Khintchin's inequality for any $p>2$. Then, by Stechkin's  result of \cite{KaSa} (chap 9.4), we conclude that  $\{t_k(x)\}$ is an unconditional convergence system.
\end{proof}
\begin{theorem}\label{T5}
	If $\phi_k$ is an orthogonal system of random variables and $\|\phi_k\|_\infty\le M$, then for any $\lambda>1$ one can find a subsequence of integers $n_k$ such that $n_{k}\le \lambda^k$ for $k\ge k(\lambda)$ and $\{\phi_{n_k}\}$ is sub-Gaussian.
\end{theorem}
The following statement is a version of a lemma from \cite{Kar3}.
\begin{lemma}\label{L5}
	Let $\phi_k$, $k=1,2,\ldots,n$ be an orthogonal system of random variables with $\|\phi_k\|_2\le 1$, and $f_j\in L^2$, $j=1,2,\ldots,m$. Then there is an $l$, $1\le l\le n$, such that
	\begin{equation}
	\sum_{j=1}^m\left|{\textbf E}(f_j\cdot \phi_l)\right|\le \sqrt{\frac{ m\cdot \sum_{j=1}^m\|f_j\|_2^2}{ n}}.
	\end{equation}
\end{lemma}
\begin{proof}
	Parseval's inequality implies
	\begin{equation}
	\sum_{j=1}^m\sum_{k=1}^n\left|{\textbf E}(f_j\cdot \phi_k)\right|^2\le \sum_{j=1}^m\|f_j\|_2^2.
	\end{equation} 
Thus there exists an $l$ such that
\begin{equation}
	\sum_{j=1}^m\left|{\textbf E}(f_j\cdot \phi_l)\right|^2\le \frac{ \sum_{j=1}^m\|f_j\|_2^2}{ n}
\end{equation}
and so by H\"{o}lder's inequality we get
\begin{equation}
\sum_{j=1}^m\left|{\textbf E}(f_j\cdot \phi_l)\right|\le \sqrt m\left(\sum_{j=1}^m\left|{\textbf E}(f_j\cdot \phi_l)\right|^2\right)^{1/2}\le \sqrt{\frac{ m\cdot \sum_{j=1}^m\|f_j\|_2^2}{ n}}.
\end{equation}

\end{proof}
\begin{proof}[Proof of Theorem \ref{T5}]
Without loss of generality we can suppose that $\|\phi_k\|_\infty\le 1$. First let us prove that there exists a sub-Gaussian subsequence $\phi_{n_k}$ such that 
\begin{equation}\label{b8}
 8^{k-1}\le n_k< 8^{k},\quad k=1,2,\ldots.
\end{equation}
We will chose $n_k$ recursively. Set $n_1=1$ and suppose that we have already chosen $n_k$, $k=1,2,\ldots, m$. Apply Lemma \ref{L5} as follows. As the collection $f_k$ we consider all possible products of functions $\phi_{n_k}$, $k=1,2,\ldots,m$. The number of such products is $2^m-1$. So applying Lemma \ref{L5}, we find $\phi_{n_{m+1}}$, $16^m\le n_{m+1}<16^{m+1}$, such that
\begin{align}
\sum_{1\le k_1<\ldots<k_l\le m}\left|{\textbf E}\left(\prod_{j=1}^l\phi_{n_{k_j}}\cdot \phi_{n_{n_{m+1}}}\right)\right|\le \sqrt{\frac{ 2^m-1}{ 8^{m+1}-8^m}}< \frac{1}{2^m}.
\end{align}
Clearly, with this we determine a quasi-multiplicative and so sub-Gaussian system $\phi_{n_k}$ satisfying \eqref{b8}. Then, observe that if the statement of theorem is satisfied for a $\lambda_1>1$ then it will hold also for $\lambda=\lambda_1^{2/3}$. Indeed, we can apply the case of $\lambda=\lambda_1$ to the systems $\phi_{2k}$ and $\phi_{2k-1}$. As a result we find sub-Gaussian subsequences $\phi_{2n_k}$ and $\phi_{2m_k-1}$ such that $n_k<\lambda_1^k$ , $m_k<\lambda_1^k$ for $k>k_0$. Letting $\{r_k\}$ to be the union of sequences $\{n_k\}$ and $\{m_k\}$ arranged in the increasing order of the terms, we consider a new sequence of random variables $\phi_{r_k}$. Clearly, it will be sub-Gaussian and one can easily check that $r_k<(\lambda_1^{2/3})^k$ for $k>2k_0$. Thus, starting with $\lambda=8$ we can prove the theorem for parameters $\lambda=8^{(2/3)^k}$, $k=1,2,\ldots,$ and so for arbitrary $\lambda>1$.
\end{proof}
A wide class of multiplicative systems was recently introduced by Rubinshtein \cite{Rub}, who has shown that the system $\phi(2^kx)$ on $[0,1)$ is multiplicative whenever $\phi$ is $1$-periodic function on the real line and on $[0,1)$ it can be written in the form
\begin{equation}\label{z5}
\phi(x)=\left\{\begin{array}{lrl}
&f(x) \hbox{ if }& x\in [0,1/4),\\
&f(1/2-x)\hbox{ if }&x\in [1/4,1/2),\\
&f(x-1/2)\hbox{ if }& [1/2,3/4),\\
&f(1-x)\hbox{ if }& [3/4,1),
\end{array}
\right.
\end{equation}
for some $f\in L^\infty[0,1/4)$. Thus, from Corollary \ref{C1} and Corollary \ref{C6} it follows that
\begin{corollary}
	For any random variable $\phi$ of the form \eqref{z5} the sequence $\phi_k(x)=\phi(2^kx)$ satisfies inequalities \eqref{a33} and \eqref{a34} (with $A_k=-1,B_k=1$).
\end{corollary}

\bibliographystyle{plain}


\end{document}